\documentclass[11pt]{amsart}
\usepackage{amscd,amsmath,amssymb,amsfonts,verbatim}
\usepackage[cmtip, all]{xy}

\newtheorem{thm}{Theorem}[section]
\newtheorem{prop}[thm]{Proposition}
\newtheorem{lem}[thm]{Lemma}
\newtheorem{cor}[thm]{Corollary}

\newtheorem{ques}[thm]{Question}
\theoremstyle{definition}

\newtheorem{defn}[thm]{Definition}
\theoremstyle{remark}
\newtheorem{remk}[thm]{Remark}
\newtheorem{remks}[thm]{Remarks}

\newtheorem{exm}[thm]{Example}
\newtheorem{exms}[thm]{Examples}
\newtheorem{notat}[thm]{Notation}
\numberwithin{equation}{section}

{\hfill$\square$\end{defn}}
{\hfill$\square$\end{remk}}
{\hfill$\square$\end{remks}}
{\hfill$\square$\end{exm}}
{\hfill$\square$\end{exms}}
{\hfill$\square$\end{notat}}

\newcommand{\thmref}{Theorem~\ref}
\newcommand{\propref}{Proposition~\ref}
\newcommand{\corref}{Corollary~\ref}

\newcommand{\lemref}{Lemma~\ref}

\newcommand{\sC}{{\mathcal C}}

\newcommand{\sF}{{\mathcal F}}
\newcommand{\sG}{{\mathcal G}}

\newcommand{\sI}{{\mathcal I}}

\newcommand{\sK}{{\mathcal K}}

\newcommand{\sO}{{\mathcal O}}

\newcommand{\sR}{{\mathcal R}}

\newcommand{\sU}{{\mathcal U}}

\newcommand{\sZ}{{\mathcal Z}}

\newcommand{\C}{{\mathbb C}}

\newcommand{\F}{{\mathbb F}}

\renewcommand{\P}{{\mathbb P}}

\newcommand{\Z}{{\mathbb Z}}

\newcommand{\fm}{{\mathfrak m}}

\newcommand{\CH}{{\rm CH}}

\newcommand{\surj}{\twoheadrightarrow}
\newcommand{\inj}{\hookrightarrow}

\newcommand{\Pic}{{\rm Pic}}

\newcommand{\Spec}{{\rm Spec \,}}

\newcommand{\ds}{{/\kern-3pt/}}

\newcommand{\ov}{\overline}

\renewcommand{\dim}{\text{\rm dim}}

\newcommand{\tuborg}{\left\{\begin{array}{ll}}
\newcommand{\sluttuborg}{\end{array}\right.}

\newcommand{\wt}{\widetilde}
\newcommand{\wh}{\widehat}

\begin{document}
\title[0-cycles and class field theory]{0-cycles on singular schemes and class 
field theory}
\author{Amalendu Krishna}
\address{School of Mathematics, Tata Institute of Fundamental Research,  
1 Homi Bhabha Road, Colaba, Mumbai, India}
\email{amal@math.tifr.res.in}



\maketitle


\begin{abstract}
We show that the Chow group of 0-cycles on a singular projective scheme $X$
over a finite field $k$ describes the abelian extensions of its function field
which are unramified over $X_{\rm sm}$. As a consequence, we obtain the
Bloch-Quillen formula for the Chow group of 0-cycles on such
schemes. We deduce simple proofs of results of Kerz-Saito for a 
class of surfaces without any assumption on ${\rm char}(k)$.
\end{abstract} 

\section{Introduction}
The aim of the class field theory in the geometric case is to describe
abelian extensions of a finitely generated field extension of a finite field in
terms of motivic invariants.
One knows that a certain class of the Galois extensions of the function field 
of a normal variety can be described in terms of the finite {\'e}tale covers 
of the same normal variety. Hence the problem of class field theory for 
normal varieties reduces to describing the abelianized {\'e}tale fundamental 
group of a normal variety in terms of its motivic invariants. 

Let $k$ be a finite field of order $q = p^a$ for a prime number $p$.
Let $X$ be a geometrically connected quasi-projective scheme over $k$ and let
$\ov{X}$ denote its base change to a fixed algebraic closure $\ov{k}$.
Then there is a short exact sequence of profinite groups:
\begin{equation}\label{eqn:Fund}
0 \to \pi^{\rm ab}_1(X)^0 \to \pi^{\rm ab}_1(X) \to \wh{\Z} \to 0.
\end{equation}

Let $X$ be a smooth projective connected scheme over $k$.
Let $\CH_0(X)$ denote the Chow group of 0-cycles 
and let $A_0(X)$ denote the kernel of the degree map 
${\rm deg}_X: \CH_0(X) \to \Z$.

The following is the main theorem of the geometric case of the class field 
theory for smooth projective schemes. The $d =1$ case was
earlier proven by Artin.
 
\begin{thm}$($\cite[Theorem~1]{KS}$)$\label{thm:CFT-KS}
Let $X$ be a smooth projective geometrically integral scheme over $k$.
Then the map $\theta^0_X: A_0(X) \to \pi^{\rm ab}_1(X)^0 $ is an isomorphism of 
finite groups.
\end{thm} 

\subsection{Main results}
Let us now assume that $X$ is a geometrically integral projective scheme
over $k$ which is not necessarily smooth. 
Let $\CH_0(X)$ denote the Chow group of 0-cycles in the sense 
of Levine-Weibel \cite{LW} (see below). Let $A_0(X)$ denote the kernel of the
degree map ${\rm deg}_X: \CH_0(X) \to \Z$. We shall say that $X$ has only
isolated singularities if the singular locus $X_{\rm sing}$ of $X$ is finite.
Let $j: X_{\rm reg} \inj X$ denote the inclusion of the regular locus of $X$.
We prove the following result in this note.

\begin{thm}\label{thm:CFT-Sing}
Let $X$ be a geometrically integral projective scheme
over $k$ of dimension $d \ge 2$ which is regular in codimension one.
Then there exists a reciprocity map 
\[
\theta_X: \CH_0(X) \to \pi^{\rm ab}_1(X_{\rm reg})
\]
which restricts to a map $\theta^0_X: A_0(X) \to \pi^{\rm ab}_1(X_{\rm reg})^0$.
The map $\theta_X$ is injective with dense image and $\theta^0_X$ is an 
isomorphism of finite groups, if $X$ has only isolated singularities.
\end{thm}

The above result extends \thmref{thm:CFT-KS} to schemes with isolated
singularities and improves the main result of \cite{SS}, where the
tame quotient of $\pi^{\rm ab}_1(X_{\rm reg})$ was described using a
quotient of $A_0(X)$. We shall show using an example that it is 
necessary to assume regularity in codimension one in order to
construct the reciprocity map $\theta^0_X$ and prove its isomorphism.
As an immediate consequence of ~\eqref{eqn:Fund} and \thmref{thm:CFT-Sing},
we get the following.
\begin{cor}\label{cor:Chow-disc}
Let $X$ be a geometrically integral projective scheme of dimension
$d \ge 2$ over a finite field. Assume that $X$ has only isolated
singularities. Assume further that $\ov{X}_{\rm reg}$ is simply connected.
Then $\CH_0(X) \xrightarrow{\simeq} \Z$.
\end{cor}

Let ${\rm Alb}_X$ denote the Albanese variety of $X$ (see \cite[Chapter~2,
\S~3]{Lang-2}). As another consequence of \thmref{thm:CFT-Sing},
one obtains the following.

\begin{cor}\label{cor:finite-Chow}
Let $X$ be a geometrically integral projective scheme of dimension
$d \ge 1$ over a finite field. Assume that $X$ has only isolated
singularities. Then $AJ_X: A_0(X) \to {\rm Alb}_X(k)$ is a surjective map
of finite groups.
\end{cor}


\subsection{Results of Kerz-Saito}
As further application of \thmref{thm:CFT-Sing}, 
we obtain simple proofs of the main results (Theorems~II and III)
of \cite{KeS} for the following  class of smooth surfaces 
without any assumption on the characteristic of $k$.

Let $X$ be a normal projective scheme over $k$ and let $U \inj X$ be an open 
subset which is smooth over $k$. Given an effective Cartier divisor
$D \subsetneq X$ with $|D| \subseteq X \setminus U$, let $\CH_0(X,D)$ denote 
the Chow group of 0-cycles on $X$ with modulus $D$ in the sense of \cite{KeS}.
Let $C(U) = {\underset{D}\varprojlim} \ \CH_0(X,D)$, where the
limit is taken over all effective Cartier divisors which are disjoint from $U$.
It is known that $C(U)$ depends only on $U$ and not on the choice of its
compactification $X$ (see \cite[Lemma~3.1]{KeS}).
Let $C_0(U)$ denote the kernel of the degree map $C(U) \to \Z$.

\begin{thm}\label{thm:KES-surf}
Let $U$ be the regular locus of a normal projective surface over $k$.
Then there is a reciprocity map $C(U) \to \pi^{\rm ab}(U)$ which 
induces an isomorphism of finite groups
$\rho_U: C_0(U) \xrightarrow{\simeq} \pi^{\rm ab}_1(U)^0$.
\end{thm}

This result was proven (without finiteness assertion) 
in \cite{KeS} for any smooth surface $U$
over $k$. However, the proof given there has a high level of complexity
and works under the assumption that ${\rm char}(k) \neq 2$. 

\subsection{Bloch-Quillen formula}
As another byproduct of the class field theory for singular projective schemes,
we obtain the following Bloch-Quillen formula for the Chow group of
0-cycles on projective schemes over a finite field which have only isolated
singularities.

\begin{thm}\label{thm:B-formula}
Let $X$ be a geometrically integral projective scheme of dimension
$d$ over a finite field. Assume that $X$ has only isolated
singularities. Then there are canonical isomorphisms
\[
\CH_0(X) \xrightarrow{\simeq} H^d(X, \sK^M_{d,X}) 
\xrightarrow{\simeq} H^d(X, \sK_{d,X}).
\]
\end{thm} 

We remark here that the Bloch-Quillen type formula for the 
Chow group of 0-cycles in known to be false if we allow non-isolated
singularities in dimension three or more (see \cite[\S~3.2]{Srinivas}).

\vskip .2cm

\subsection{Outline}
Our strategy for proving \thmref{thm:CFT-Sing} is to first prove it for  
surfaces. This case requires us to use a result of Kato and Saito
\cite{KS-2} which describes the class field theory of regular open
subsets of projective schemes in terms of a generalized id{\`e}le class group.
The general case is deduced from surfaces using induction
on dimension with the aid of Lefschetz type theorems for fundamental groups
and Bertini type theorems over finite fields. 
Theorem~\ref{thm:KES-surf} is proven by combining \thmref{thm:CFT-Sing}
with some results of \cite{KSri} and cycle class maps for Chow
groups with modulus. We prove Theorems~\ref{thm:B-formula} 
using a combination of \thmref{thm:CFT-Sing} and main results of \cite{KS-2}.

\section{Chow group of 0-cycles on singular schemes}
\label{sec:CFT*-0}
In this section, we recall the definition of the Chow group of 0-cycles
for singular schemes from \cite{LW}. Specializing to the schemes with
isolated singularities, we show that there exist canonical
maps from this Chow group to the top Zariski cohomologies of the Milnor
and the Quillen $K$-sheaves.

\subsection{Chow group of 0-cycles}\label{section:LW}
We first define the Chow group of 0-cycles for curves. 
A {\sl curve} will mean an equi-dimensional
quasi-projective scheme over $k$ of dimension one.

Let $C$ be a reduced curve and let $D \subsetneq C$ be a closed subscheme 
such that $C_{\rm sing} \subseteq D$ and $D$ contains no irreducible component 
of $C$. Let $\sZ_0(C,D)$ denote the free abelian group on closed points in 
$C \setminus D$. Let $\{C_1, \cdots , C_r\}$ denote the set of irreducible
components of $C$ with generic points $\{\eta_1, \cdots , \eta_r\}$,
respectively. Let $k(C)$ denote the ring of total quotients on $C$.
Since $C$ is a reduced curve, it is Cohen-Macaulay. This implies in particular
that the map $k(C) \to \stackrel{r}{\underset{i=1}\prod} \sO_{C, \eta_i}$ is
an isomorphism and hence the map $\iota_{(C,D)}: \sO^{\times}_{C,D} \to 
\stackrel{r}{\underset{i=1}\prod} \sO_{C, \eta_i}$ is injective.
 
Given $f \in \sO^{\times}_{C,D}$, let $(f)_C$ denote the divisor of zeros and 
poles of $f$ on $C$ in the sense of \cite{Fulton}. 
Since $C_{\rm sing} \subseteq D$, it is clear from the above definitions that
$(f)_C = \stackrel{r}{\underset{i=1}\sum} (f_i)_{C_i}$. Let $\sR_0(C,D)$
denote the subgroup of $\sZ_0(C,D)$ generated by the set 
$\{(f)_C| f \in \sO^{\times}_{C,D}\}$. The group $\sR_0(C,D)$ is called the
group of rational equivalences on $C$ relative to $D$.
The {\sl Chow group of 0-cycles} on $(C,D)$ is defined
\nolinebreak
by 
\[
\CH_0(C,D):= \frac{\sZ_0(C,D)}{\sR_0(C,D)}.
\]
 
\begin{prop}$($\cite[Proposition~1.4]{LW}$)$\label{prop:0-cyc-curve}
Let $(C,D)$ be as above. Then there is a cycle class map $cyc_C: 
\CH_0(C,D) \to K_0(C)$ which induces an isomorphism $cyc_C: \CH_0(C,D) 
\xrightarrow{\simeq} \Pic(C)$. In particular, $\CH_0(C,D)$ is independent of 
$D$.
\end{prop}

Let $X$ now be a reduced and connected equi-dimensional quasi-projective 
scheme over $k$ of dimension $d \ge 2$. A reduced curve $C \inj X$ is called 
{\sl Cartier} if: 
\begin{enumerate}
\item
Each component of $C$ intersects $X_{\rm sing}$ properly.
\item
For every closed point $x \in X_{\rm sing} \cap C$, the inclusion
$C \inj X$ is a local complete intersection at $x$. 
\end{enumerate}

\enlargethispage{25pt}

Let $\sZ_0(X)$ denote the free abelian group on the set of closed points
in $X_{\rm reg}$. Given a Cartier curve $\iota: C \inj X$, let $D = C \cap
X_{\rm sing}$ and let $\sO^{\times}_{C,D} \inj k(C) \simeq 
\stackrel{r}{\underset{i=1}\prod} \sO_{C, \eta_i}$ be the inclusion as 
before. Here, $\{C_1, \cdots , C_r\}$ denotes the set of irreducible
components of $C$ with generic points $\{\eta_1, \cdots , \eta_r\}$.
 
Given $f \in \sO^{\times}_{C,D}$ and $1 \le i \le r$, 
let $f_i$ denote the $i$th component
of $f$ in $k(C)$ and let $(f_i)_{C_i}$ denote the divisor of zeros and 
poles of $f_i$ on $C_i$ in the sense of \cite{Fulton}. 
We set $(f)_C := \stackrel{r}{\underset{i=1}\sum} \iota_*((f_i)_{C_i})$. 
Let $\sR_0(X)$ denote the subgroup of $\sZ_0(X)$ generated by the set 
$\{(f)_C| f \in \sO^{\times}_{C,D}, \ C \ \ \mbox{Cartier \ on} \ X\}$. 
The group $\sR_0(X)$ will be called the
group of rational equivalences.
The {\sl Chow group of 0-cycles} on $X$ is defined
by 
\begin{equation}\label{eqn:0-cycl}
\CH_0(X):= \frac{\sZ_0(X)}{\sR_0(X)}.
\end{equation}

The following result simplifies the definition of rational equivalence in
special cases. 

\begin{lem}$($\cite[\S~2]{BS-1}$)$\label{lem:RE}
In the definition of the group of rational equivalences of 0-cycles
above, we can assume that a Cartier curve $C$ is irreducible 
if $X$ is so.
\end{lem}

\begin{exm}\label{exm:failure}
Let $C$ be the projective plane curve over $k$ which has a simple cusp along
the origin and is regular elsewhere. Its local ring at the singular point
is analytically isomorphic to $k[[t^2, t^3]]$ which is canonically a 
subring of its normalization $k[[t]]$.
Let $\pi: \P^1 \to C$ denote the normalization map. Let $S \simeq
\Spec({k[t]}/{(t^2,t^3)})$ denote the reduced conductor and let
$\wt{S} \simeq \Spec({k[t]}/{(t^2)})$ denote its scheme-theoretic inverse image
in $\P^1$. We then have a commutative diagram with exact rows:
\begin{equation}\label{eqn:failure-0}
\xymatrix@C1pc{
0  \ar[r] & {\underset{m}\varprojlim} {\sO^{\times}(mS)}/{k^{\times}} 
\ar[r] \ar[d] & {\underset{m}\varprojlim} K_0(C, mS) \ar[r] \ar[d]_{\simeq} &
\Pic(C) \ar[r] \ar[d]^{\pi^*} & 0 \\
0  \ar[r] & {\underset{m}\varprojlim} {\sO^{\times}(m\wt{S})}/{k^{\times}} 
\ar[r] & {\underset{m}\varprojlim} K_0(\P^1, m\wt{S}) \ar[r]  &
\Pic(\P^1) \ar[r] & 0.}
\end{equation}
The isomorphism of the middle vertical map follows from the known result
that the double relative $K$-groups $K_0(C, \P^1, mS)$ and $K_{-1}(C, \P^1, mS)$
vanish. 

It is easy to check from the $K$-theory localization sequence
that $\Pic(\P^1, m\wt{S}) \xrightarrow{\simeq} K_0(\P^1, m\wt{S})$.
On the other hand, the known class field theory for curves (with modulus)
tells us that there is a canonical isomorphism
${\underset{m}\varprojlim} \Pic^0(\P^1, m\wt{S}) \xrightarrow{\simeq} 
\pi^{\rm ab}(C_{\rm reg})^0$.
It follows that there is an isomorphism
$(1 + tk[[t]])^{\times} \xrightarrow{\simeq} \pi^{\rm ab}(C_{\rm reg})^0$.
On the other hand, $A_0(C) \simeq \Pic^0(C) \simeq k$ is finite which shows 
that there is no reciprocity map $A_0(C) \to \pi^{\rm ab}(C_{\rm reg})^0$ 
and the two can not be isomorphic. 
\end{exm}

\subsection{0-cycles and cohomology of $K$-theory sheaves}
\label{sec:Coh-K}
In this note, all cohomology groups will be with respect to the Zariski
topology unless we explicitly mention otherwise.
Given a $k$-scheme $X$, let $\sK^M_{m,X}$ denote the Zariski sheaf on $X$
whose stalk at a point $P \in X$ is the Milnor $K$-group $K^M_m(\sO_{X,P})$
of the local ring $\sO_{X,P}$ (see \cite[\S~1.3]{KS}).
The corresponding sheaf of Quillen $K$-theory will be denoted by $\sK_{m,X}$.
Let us now assume that $X$ is a quasi-projective scheme of
dimension $n \ge 1$ over $k$ with only singularities and let $S$ denote the 
singular locus of $X$. For $j \ge 0$, let $X^j$ denote the set of 
codimension $j$ points on $X$.
Consider the maps of Zariski sheaves:
\begin{equation}\label{eqn:Milnor-K}
\sK^M_{m, X} \xrightarrow{\epsilon}
\left(\begin{array}{c}{\underset{x \in X^0}\coprod}\ (i_x)_* K^M_{m}(k(x)) 
\\ \oplus \\
{\underset{P \in S}\coprod}\ (i_P)_* K^M_{m}(\sO_{X,P}) 
\end{array}\right) \xrightarrow{d_0}
\left(\begin{array}{c}
{\underset{x \in X^1}\coprod} \ (i_x)_* K^M_{m-1}(k(x)) 
\\ \oplus \\
{\underset{P \in S}\coprod}
{\underset{P \in \ov{\{x\}}}\coprod}\ (i_x)_* K^M_{m}(k(x)) 
\end{array}\right)
\xrightarrow{d_1} \cdots
\end{equation}
\[
\cdots \xrightarrow{d_{n-1}}
\left(\begin{array}{c}
{\underset{x \in X^n}\coprod} \ (i_x)_* K^M_{m-n}(k(x)) 
\\ \oplus \\
{\underset{P \in S}\coprod}
{\underset{P \in \ov{\{x\}}}\coprod}\ (i_x)_* K^M_{m-n+1}(k(x)) 
\end{array}\right)
\xrightarrow{d_n} 
\left(\begin{array}{c}
0 \\ \oplus \\
{\underset{P \in S}\coprod} \ (i_P)_* K^M_{m-n}(k(P)) 
\end{array}\right)
\to 0.
\]
Here, the map $\epsilon$ is induced by the inclusion into both terms and
the other maps are given by the matrices
\[
d_0 = \left(\begin{array}{cc}
\partial_1 & 0 \\
-\Delta & \epsilon
\end{array}\right), 
d_1 = \left(\begin{array}{cc}
\partial_1 & 0 \\
\Delta & \partial_2
\end{array}\right),
\cdots ,
d_n = \left(\begin{array}{cc}
0 & 0 \\
\Delta & \partial_2
\end{array}\right)
\]
with $\partial_1$ and $\partial_2$ being the differentials of
the Gersten-Quillen complex for Milnor $K$-theory sheaves as
described in \cite{Kato} (see also \cite{EM}) and $\Delta$'s being the 
diagonal maps.

\begin{lem}\label{lem:Milnor-res}
The above sequence of maps forms a complex which gives a flasque resolution
of the sheaf $\epsilon(\sK^M_{m,X})$.
\end{lem}
\begin{proof}
A similar complex for the Quillen $K$-theory sheaves is constructed in 
\cite[\S~5]{PW-1} and it is shown there that this complex is 
a flasque resolution of $\epsilon(\sK_{m,X})$. 
The same proof works here in verbatim. On all stalks except at $S$, the 
exactness follows from \cite[Proposition~4.3]{EM}. The
exactness at the points of $S$ is an immediate consequence of the way the
differentials are defined in ~\eqref{eqn:Milnor-K} (see
\cite{PW-1} for details). 
\end{proof}

For a scheme $X$, let $\CH^F_0(X)$ denote the Chow group of 0-cycles on $X$
in the sense of \cite{Fulton}.
Our proof of the main results will be based on the following
descriptions of the Chow groups of 0-cycles.

\begin{prop}\label{prop:Milnor-res-*}
Let $X$ be an integral quasi-projective scheme of dimension $d \ge 1$ over $k$ 
with only isolated singularities. Then there are canonical maps
\[
\CH_0(X) \to H^d(X, \sK^M_{d,X}) \to H^d(X, \sK_{d,X}) \to \CH^F_0(X),
\]
where the second map is an isomorphism.
\end{prop} 
\begin{proof}
The $d = 1$ case is already dealt with in \propref{prop:0-cyc-curve}.
We can thus assume that $d \ge 2$. Let $S$ denote the singular locus of $X$ and
let $X^{j}_S$ denote the set of points on $X$ of codimension $j$ such that
$S \cap \ov{\{x\}} = \emptyset$. 
We first observe that the map of sheaves $\sK^M_{d, X} \surj 
\epsilon(\sK^M_{d, X})$ is generically an isomorphism
and the same holds for the Quillen $K$-theory sheaves.
It follows that (see \cite[Exer.~II.1.19, Lemma~III.2.10]{Hart})
that the map $H^d(X, \sK^M_{d,X}) \to H^d(X, \epsilon(\sK^M_{d,X}))$
is an isomorphism and ditto for the Quillen $K$-theory sheaves.
It follows from \lemref{lem:Milnor-res} that both $H^d(X, \sK^M_{d,X})$ 
and $H^d(X, \sK_{d,X})$ are given by the middle homology of the complex
$\sC_X$: 
\[
\left(\begin{array}{c}
{\underset{x \in X^{d-1}}\coprod} \ K_{1}(k(x)) 
\\ \oplus \\
{\underset{P \in S}\coprod}
{\underset{P \in \ov{\{x\}}}\coprod}\ K_{2}(k(x)) 
\end{array}\right) 
\xrightarrow{d_{1}}
\left(\begin{array}{c}
{\underset{x \in X^d}\coprod} \ K_{0}(k(x)) 
\\ \oplus \\
{\underset{P \in S}\coprod}
{\underset{P \in \ov{\{x\}}}\coprod}\ K_{1}(k(x)) 
\end{array}\right)
\xrightarrow{d_2} 
\left(\begin{array}{c}
0 \\ \oplus \\
{\underset{P \in S}\coprod} \ K_{0}(k(P)) 
\end{array}\right).
\]
On the other hand, letting $\sC^0_X$ and $\sC^{F,0}_X$ denote the complexes
\[
{\underset{x \in X^{d-1}_S}\coprod} \ K_{1}(k(x)) 
\xrightarrow{{\rm div}}  {\underset{x \in X^d_S}\coprod} \ 
K_{0}(k(x)) \to 0 \ \ 
{\rm and}  
\]
\[
{\underset{x \in X^{d-1}}\coprod} \ K_{1}(k(x)) 
\xrightarrow{{\rm div}}  {\underset{x \in X^d}\coprod} \ K_{0}(k(x)) \to 0, 
\]
respectively, we see that there are canonical maps of chain complexes
$\sC^0_X \inj \sC_X \surj \sC^{F,0}_X$. This 
yields canonical maps $H_1(\sC^0_X) \to H^d(X, \sK^M_{d,X}) \xrightarrow{\simeq} 
H^d(X, \sK_{d,X}) \to H_1(\sC^{F,0}_X)$. 
It follows however by using the definition of 
$\CH_0(X)$ and \lemref{lem:RE} that $H_1(\sC^0_X) \simeq \CH_0(X)$.
It is also clear that $H_1(\sC^{F,0}_X) \simeq \CH^F_0(X)$.
This finishes the proof.
\end{proof}

\subsection{Cycle class map for singular schemes}\label{sec:CCS}
For a scheme $X$ of dimension $d \ge 2$ and a closed subscheme $D \subseteq X$,
let $K(X,D)$ denote the {\sl relative $K$-theory} spectrum given by
the homotopy fiber of the restriction map of spectra $K(X) \to K(D)$.
Given a closed point $x \in X_{\rm reg} \setminus D$, the map of pairs
$(\Spec(k(x)), \emptyset) \to (X,D)$ is of finite tor-dimension.
Hence, it yields a map
$K_0(k(x)) \to K_0(X,D)$ which allows us to define a cycle class $cyc([x]) \in
K_0(X,D)$ as the image of $1 \in K_0(k(x))$.
Let $F^dK_0(X,D)$ denote the subgroup of $K_0(X,D)$ generated by the
cycle classes of all closed points in $X_{\rm reg} \setminus D$.
Let $F^dK_0(X)$ denote the subgroup of $K_0(X)$ generated by cycle classes
of closed points of $X_{\rm reg}$. 
It follows from \cite[Proposition~2]{LW} that the cycle class map
$cyc_X: \sZ_0(X) \to F^dK_0(X)$ kills the group of rational equivalences to 
yield a surjective map
\begin{equation}\label{eqn:Cyc-class-Sing}
cyc_X: \CH_0(X) \surj F^dK_0(X).
\end{equation}

\begin{lem}\label{lem:CCS-0}
Suppose that $X$ has only isolated singularities and let $S \inj X$ be
a closed subscheme supported on $X_{\rm sing}$. Then the maps
$\sZ_0(X) \to K_0(X, S) \to K_0(X)$ induce surjective maps
\begin{equation}\label{eqn:CCS-0-0}
\CH_0(X) \surj F^dK_0(X,S) \surj F^dK_0(X)
\end{equation}
such that the second map is an isomorphism. 
\end{lem}
\begin{proof}
In any case, we have the maps $\sZ_0(X) \surj F^dK_0(X,S) \surj F^dK_0(X)$
and it follows from ~\eqref{eqn:Cyc-class-Sing} that the composite map kills
rational equivalences. So it suffices to show that the map
$F^dK_0(X,S) \to F^dK_0(X)$ is an isomorphism. 
But this follows from \cite[Lemma~3.1]{AK-1}.
\end{proof}

\section{Class field theory for surfaces}\label{sec:Surf*}
In this section, we prove our main results for surfaces.
\thmref{thm:CFT-Sing} is proven by combining the results of 
\cite{KS-2} and the isomorphism of the canonical map $\CH_0(X) \to
H^2(X, \sK^M_{2,X})$. \thmref{thm:KES-surf} is proven with the aid of
\thmref{thm:CFT-Sing} and an explicit formula for the Chow group
of 0-cycles on normal surfaces.

For a $k$-scheme $X$, let $\wt{K}_0(X) := 
{\rm Ker}({\rm rk}: K_0(X) \to H^0(X, \Z))$ and let
$SK_0(X) = {\rm Ker}({\rm det}:\wt{K}_0(X) \to \Pic(X))$.
Let us now assume that $X$ is a surface with isolated singularities.
If $x \in X_{\rm reg}$, then $\sO_{\{x\}}$ is locally defined by a 
regular sequence of length two. It follows that $[\sO_{\{x\}}] \in K_0(X)$ lies
in $SK_0(X)$. We conclude that the cycle class map
$cyc_X: \CH_0(X) \to K_0(X)$ factors through $\CH_0(X) \surj F^2K_0(X) \inj
SK_0(X)$. 

\begin{prop}\label{prop:Surface-0}
Let $X$ be a quasi-projective surface over $k$ with isolated singularities
and let $S \inj X$ be a closed subscheme supported on $X_{\rm sing}$.
Then the following hold.
\begin{enumerate}
\item
The maps $\CH_0(X) \to F^2K_0(X,S) \to F^2K_0(X) \inj SK_0(X)$
are all isomorphisms.
\item
There are isomorphisms
\[
F^2K_0(X,S) \simeq F^2K_0(X) \simeq H^2_{\sC}(X, \sK_2) \simeq
H^2_{\sC}(X, \sK_{2, (X,S)}) 
\]
for $\sC = Zar$ or $Nis$.
\item
If $\pi: \wt{X} \to X$ is the normalization, then the maps 
$\pi^*: \CH_0(X) \to \CH_0(\wt{X})$ and $\pi^{\rm ab}_1(X_{\rm reg}) \to 
\pi^{\rm ab}_1(\wt{X}_{\rm reg})$ are isomorphisms.
\end{enumerate}
\end{prop}
\begin{proof}
Since the first two maps in (1) are anyway surjective, it suffices to show that
the composite map $\CH_0(X) \to SK_0(X)$ is an isomorphism.
Since $X$ is a surface with isolated singularities, 
it follows from the above definition of
Cartier curves and \lemref{lem:RE} that we can assume that the group of rational
equivalences in $\sZ_0(X)$ is generated by the divisors of rational functions 
on integral curves on $X$ which do not meet $X_{\rm sing}$. 
It follows now from \cite[Theorem~2.2]{PW-1} that there is an exact sequence
\[
SK_1(X_{\rm sing}) \to \CH_0(X) \xrightarrow{cyc_X} SK_0(X) \to 0.
\]
As $X_{\rm sing}$ is finite, the first term in this sequence is zero
and we get (1).

We now show (2). The isomorphism $SK_0(X) \simeq  H^2_{\sC}(X, \sK_2)$ follows
directly from the Thomason-Trobaugh spectral sequence
$E^{p,q}_2 = H^p_{\sC}(X, \sK_q) \Rightarrow K_{q-p}(X)$ with differential
$d_r: E^{p,q}_r \to E^{p+r,q+r-1}_r$ and the fact that the
map $K_1(X) \to H^0_{\sC}(X, \sK_1) \simeq \sO^{\times}(X)$ is split surjective
(see \cite[pg. 162]{KSri}).

We are left with showing that the map 
$H^2_{\sC}(X, \sK_{2, (X,S)}) \simeq H^2_{\sC}(X, \sK_2)$ is an isomorphism
for $\sC = Zar$ or $Nis$.
We have an exact sequence of sheaves
\[
0 \to {\sK_{3,S}}/{\sK_{3,X}} \to \sK_{2, (X,S)} \to \sK_{2,X} \to
\sK_{2,S} \to 0.
\]
This yields two short exact sequences 
$0 \to {\sK_{3,S}}/{\sK_{3,X}} \to \sK_{2, (X,S)} \to
\sF \to 0$ and $0 \to \sF \to \sK_{2,X} \to
\sK_{2,S} \to 0$. 

Since ${\sK_{3,S}}/{\sK_{3,X}}$ and $\sK_{2,S}$ are supported on $S$ which is
0-dimensional, the isomorphism 
$H^2_{Zar}(X, \sK_{2, (X,S)}) \simeq H^2_{Zar}(X, \sK_2)$ follows at once by 
considering the long cohomology exact sequences.
To prove this for the Nisnevich site, we apply the same argument
using the following inputs: if $\sG$ is a sheaf supported on $i:S \inj X$, then
$\sG \xrightarrow{\simeq} i_*(\sG|_S)$. In particular,
$H^i_{Nis}(X, \sG) = 0$ for $i > 0$ because $i_*$ is exact in Nisnevich
topology and the Nisnevich cohomological dimension of $S$ is zero.
This finishes the proof of (2).

The isomorphism $\CH_0(X) \xrightarrow{\simeq} \CH_0(\wt{X})$ follows
directly from (1), (2) and \cite[Proposition~2.3]{AK-2}.
It should be observed here that the relevant part of the cited result holds
over any field even though the result there is stated over $\C$.
The isomorphism $\pi^{\rm ab}_1(X_{\rm reg}) \xrightarrow{\simeq} 
\pi^{\rm ab}_1(\wt{X}_{\rm reg})$ follows from the Zariski-Nagata
purity theorem (see \cite[Expos{\'e}~X, Th{\'e}or{\`e}me~3.1]{SGA-1})
because $X_{\rm reg} \inj \wt{X}_{\rm reg}$ is an open immersion of smooth
schemes whose complement has codimension two. This proves (3).
\end{proof}

\begin{cor}\label{cor:Surface-0-*1}
Let $X$ be as in \propref{prop:Surface-0} and let $\sI_S$ denote
the sheaf of ideals defining a closed subscheme $i: S \inj X$
supported on the singular points.
Let $\sK^M_{2,(X,S)}$ denote the Zariski sheaf defined by the exact sequence
\[
0 \to \sK^M_{2,(X,S)} \to  \sK^M_{2,X} \to i_*(\sK^M_{2,S}) \to 0.
\]
Then the maps $\CH_0(X) \to H^2(X, \sK^M_{2,X})$ and 
$H^2(X, \sK^M_{2,(X,S)}) \to  H^2(X, \sK^M_{2,X})$ are
isomorphisms.
\end{cor}
\begin{proof}
The first isomorphism follows by combining 
Propositions~\ref{prop:Milnor-res-*} and ~\ref{prop:Surface-0}.
For the second isomorphism, observe that the map
$\sK^M_{2,X} \to \sK_{2,X}$ is an isomorphism. In particular, the map
$\sK_{2, (X,S)} \to \sK^M_{2,(X,S)}$ is surjective and an isomorphism away
from $S$. It follows that the map $H^2(X, \sK_{2,(X,S)}) \to
H^2(X, \sK^M_{2,(X,S)})$ is an isomorphism. We now apply 
\propref{prop:Surface-0} again. 
\end{proof}

Let $X$ be a connected and projective scheme over $k$ of dimension $d \ge 2$
which is regular in codimension one.
Let ${\rm deg}: \CH_0(X) \to \Z$ denote the degree map. 
Observe that this degree map is surjective.
The reason is that we can find a smooth curve $C \inj X$ which does not
meet $X_{\rm sing}$ and it is a consequence of Lang's density theorem
\cite{Lang-1} that the map ${\rm deg}: \CH_0(C) \to \Z$ is surjective.
Let $A_0(X)$ denote the kernel of the degree map.

\enlargethispage{25pt}

\begin{thm}\label{thm:Main-Surf}
Let $X$ be a geometrically integral and projective surface over $k$ with
isolated singularities. Let $U \inj X$ denote the regular locus of $X$.
Then there exists a reciprocity map
$\theta_X: \CH_0(X) \to \pi^{\rm ab}_1(X)$ which induces an isomorphism
$\theta^0_X: A_0(X) \xrightarrow{\simeq} \pi^{\rm ab}_1(U)^0$.
\end{thm}
\begin{proof}
To construct the reciprocity map 
$\theta_X: \CH_0(X) \to \pi^{\rm ab}_1(U)$,
we let $x \in U$ be a closed point and consider the inclusion
$\iota^x: \Spec(k(x)) \inj U$. This induces a natural map
$\iota^x_*: \pi^{\rm ab}_1(\Spec(k(x))) \to \pi^{\rm ab}_1(U)$.
We set $\theta_X([x]) = \iota^x_*(F_x)$, where $F_x$ is the Frobenius
element of $\pi^{\rm ab}_1(\Spec(k(x))) \simeq {\rm Gal}({\ov{k(x)}}/{k(x)})$.
Extending linearly, we get the reciprocity map
\begin{equation}\label{eqn:Rec-Normal}
\theta_X: \sZ_0(X) \to \pi^{\rm ab}_1(U).
\end{equation}

To show that $\theta_X$ kills rational equivalences,
recall from \cite[\S~3]{KS-2} that there is a reciprocity
map $\rho_X: {\underset{S \subseteq X_{\rm sing}}\varprojlim} \ 
H^2(X, \sK^M_{2,(X,S)}) \to \pi^{\rm ab}_1(U)$.
It follows from \corref{cor:Surface-0-*1} that this map factors through  
a reciprocity map $\rho_X: H^2(X, \sK^M_{2,X}) \to \pi^{\rm ab}_1(U)$.

We now consider the maps $\sZ_0(X) \xrightarrow{\eta_X} H^2(X, \sK^M_{2,X}) 
\xrightarrow{\rho_X} \pi^{\rm ab}_1(U)$, where the first map is the composite
$\sZ_0(X) \surj \CH_0(X) \to H^2(X, \sK^M_{2,X})$ given by 
\propref{prop:Milnor-res-*}. To prove that $\theta_X$ kills rational
equivalences, it suffices therefore to show that $\theta_X = \rho_X \circ
\eta_X$. For this, it suffices to show that given any closed point
$x \in U$ and any finite separable field extension $k(X) \inj L$ in which
$U$ is unramified, the composite
$K_0(k(x)) \simeq H^2_{x}(X, \sK^M_{2,X}) \to H^2(X, \sK^M_{2,X}) \to 
\pi^{\rm ab}_1(U) \to {\rm Gal}(L/{k(X)})$ takes the element $1 \in K_0(k(x))$ 
to the Frobenius substitution $F_x$ over $x$. But this follows from
\cite[Proposition~3.6]{KS-2}.

It is clear from the above construction that $\theta_X$ induces a 
map $\theta^0_X: A_0(X) \to \pi^{\rm ab}_1(U)^0$. The assertion that this
map is an isomorphism follows from \corref{cor:Surface-0-*1} and
\cite[Theorem~6.2]{Raskind}. 
\end{proof}

\subsection{Proof of  \thmref{thm:KES-surf}}\label{sec:Main-2-prf}
Let $X$ be an integral normal projective surface over $k$ and let
$S$ denote the singular locus of $X$ with reduced induced closed subscheme
structure. Set $U = X_{\rm reg}$. 
Let $f: \wt{X} \to X$ denote a resolution of singularities 
of $X$. Recall that such a resolution of singularities exists for surfaces 
over any field.
Let $E \subsetneq \wt{X}$ denote the reduced exceptional divisor.
We identify $f^{-1}(U)$ with $U$ in what follows. 
We begin by recalling the Chow group of 0-cycles with modulus from 
\cite{KeS}.

Let $D \subsetneq \wt{X}$ be an effective Cartier divisor supported on $E$.
Let $\sZ_0(\wt{X}, D)$ denote the free abelian group
on closed points in $\wt{X} \setminus D$. Let $C \inj \wt{X} \times \P^1_k$
be a closed irreducible curve satisfying
\begin{enumerate}
\item
$C$ is not contained in $\wt{X} \times \{0, 1, \infty\}$.
\item
If $\nu: C^N \to \wt{X} \times \P^1_k$ denotes the composite map from the
normalization of $C$, then one has an inequality of Weil divisors on $C^N$:
\[
\nu^*(D \times \P^1_k) \le \nu^*(\wt{X} \times \{1\}).
\]

\end{enumerate}
We call such curves admissible. Let $\sZ_1(\wt{X}, D)$ denote the
free abelian group on admissible curves and let $\sR_0(\wt{X}, D)$ denote
the image of the boundary map 
$(\partial_{0} - \partial_{\infty}): \sZ_1(\wt{X}, D) \to \sZ_0(\wt{X}, D)$.
The Chow group of 0-cycles on $\wt{X}$ with modulus $D$ is defined as
the quotient
\[
\CH_0(\wt{X}, D) := \frac{\sZ_0(\wt{X}, D)}{\sR_0(\wt{X}, D)}.
\]

\begin{prop}\label{prop:Cyc-class-mod}
There is a cycle class map $cyc_{(\wt{X},D)}:\CH_0(\wt{X}, D) \to K_0(X,D)$.
\end{prop}
\begin{proof}
This is a special case of the more general construction of the 
cycle class map by Binda and Krishna \cite{BK}.
They give a functorial construction of the cycle class map
$cyc_{(\wt{X},D)}:\CH_0(\wt{X}|D, n) \to K_n(X,D)$ from the
higher Chow groups with modulus to the higher relative $K$-groups.
A completely different construction of this cycle class map is
also given in \cite{Binda}.
We reproduce the construction of \cite{BK} in the present special case
for the sake of completeness.

Let $x \in \wt{X} \setminus D$ be a closed point with residue field $k(x)$
and let $\iota^x: \Spec(k(x)) \inj \wt{X}$ denote the inclusion.
Then the composite map $K(\Spec(k(x))) \xrightarrow{\iota^x_*} K(\wt{X}) \to
K(D)$ is null-homotopic and hence there is a unique factorization
$\iota^x_*:K(\Spec(k(x))) \to K(\wt{X}, D) \to K(\wt{X})$.
We set $[x] = \iota^x_*(1) \in K_0(\wt{X}, D)$.  
Extending this linearly, we get a cycle class map
$cyc_{(\wt{X}, D)}: \sZ_0(\wt{X}, D) \to K_0(\wt{X}, D)$. 

Our next task is to show that this map has the desired factorization.
So let $C \subsetneq \wt{X} \times \P^1_k$ be an irreducible admissible curve.
Let $C^N$ denote the normalization of $C$ and let
$g: C^N \to \wt{X}$ be the projection map.
If $C$ lies over a closed point of $\wt{X}$, then it is immediate that 
$cyc_{(\wt{X}, D)}(\partial([C])) = 0$.
So we can assume that $C$ does not lie over a closed point.
In that case, the map $g: C^N \to \wt{X}$ is finite. 
This gives rise to a Cartesian square
\begin{equation}\label{eqn:Cyc-mod-*2}
\xymatrix@C1pc{
C^N \ar@{^{(}->}[dr]^{\phi} \ar@/_2pc/[ddr]_{{\rm id}} \ar@/^1pc/[drr]^{\nu} & & \\
& C^N \times \P^1_k \ar[r] \ar[d]_{p'} &
\wt{X} \times \P^{1}_k \ar[d]^{p} \\
& C^N \ar[r]_{g} & \wt{X}.}
\end{equation}

The finiteness of $g$ and admissibility of $C$ imply that 
$C^N = \phi(C^N)$ is an admissible cycle on $C^N \times \P^1_k$.
We set $D' = g^*(D)$. Notice that a consequence
of the admissibility condition is that $D'$ is a proper Cartier divisor on 
$C^N$. Our first claim is the following:

\vskip .3cm

\noindent
{\bf Claim:}
The finite map $g$ induces a push-forward map between $K$-theory spectra
$g_*: K(C^N,D') \to K(\wt{X}, D)$.

\vskip .1cm

{\sl Proof of the claim:}
Since $g$ is a morphism between non-singular schemes, it is of finite
tor-dimension and hence there is a push-forward map of $K$-theory spectra
$g_* : K(C^N) \to K(\wt{X})$. To prove the claim, it is enough to show
that the map $D' \to D$ is a also of finite tor-dimension.
To show this, all we need to know is that $C^N$ and $D$
are tor-independent over $\wt{X}$. 
Since $D$ is an effective Cartier divisor, only 
possible tor term will be ${\rm Tor}^1_{\sO_{\wt{X}}}(\sO_{C^N}, \sO_{D})$
which is same as $\sI_{D}$-torsion subsheaf of $\sO_{C^N}$. Since
$C^N$ is integral, this torsion subsheaf is non-zero if and only if
the ideal $\sI_{D'}$ is zero. But this can not happen as $D'$ is a proper
divisor on $C^N$. This proves the claim.

It is easy to check from the above construction that
\begin{equation}\label{eqn:RES-Norm-0}
\partial^{\epsilon}([C]) = g_*(\partial^{\epsilon}([C^N]) \ \ \mbox{for} \ \
\epsilon = 0, \infty \ \ \mbox{and}
\end{equation}
\[
cyc_{(\wt{X},D)} \circ g_*(\alpha) = g_* \circ cyc_{(C^N, D')}(\alpha) 
\ \ \mbox{for} \ \alpha \in \sZ_0(C^N, D').
\]

This reduces the problem of constructing the cycle class map to the case
when $\wt{X}$ is replaced by a smooth curve $C$ and $D$ is 
an effective Cartier divisor on $C$. 
This case follows from Lemma~\ref{lem:Cyc-class-curve} below.
\end{proof} 


\begin{lem}\label{lem:Cyc-class-curve}
Let $C$ be a smooth curve and let $D \subsetneq C$ be an effective Cartier
divisor. Then there is a cycle class map 
$cyc_{(C,D)}: \CH_0(C,D) \to K_0(C,D)$
which induces an isomorphism $\CH_0(C,D) \xrightarrow{\simeq} \Pic(C,D)$.
\end{lem}
\begin{proof}
We let $\sR'_0(C,D)$ denote the subgroup of $\sZ_0(C,D)$ generated by
the cycles ${\rm div}(f)$, where $f \in \sO^{\times}_{C,D} \inj k(C)^{\times}$
is such that $f \equiv 1$ modulo $\sI_D$.
Using ~\eqref{eqn:Cyc-mod-*2}, it is not difficult to check that 
$\sR'_0(C,D) = \sR_0(C,D)$.
So we shall use this new definition of rational equivalences.

Let $A$ denote the semi-local ring $\sO_{C,D}$ and let $I$ denote the ideal
of $D$ in $A$ giving an exact sequence
\[
0 \to K_1(A,I) \to K_1(A) \to K_1(A/I) \to 0.
\]

We now consider the commutative diagram of homotopy fiber sequences:
\[
\xymatrix@C1pc{
{\underset{x \notin D}\coprod} K(k(x)) \ar[r] & K(C) \ar[r] \ar[d] & 
K(A) \ar[d] \\
& K(A/I) \ar@{=}[r] & K(A/I).}  
\]
This yields a homotopy fiber sequence
\[
{\underset{x \notin D}\coprod} K(k(x)) \to K(C,D) \to K(A,I)
\]
and in particular, an exact sequence
\[
K_1(A,I) \xrightarrow{\partial} \sZ_0(C, D) \to K_0(C,D) \to 0
\]
and we conclude from this that 
\[
{\rm Coker}(\partial) = \CH_0(C, D) \xrightarrow{\simeq} K_0(C, D).
\]
Finally, it is easy to check that $\Pic(C, D) \simeq K_0(C,D)$. 
\end{proof}

\begin{prop}\label{prop:RES-Norm}
Let $f: \wt{X} \to X$ be a resolution of singularities of a normal
projective surface $X$ as above. 
For any $m \ge 1$, there exists a commutative diagram
\begin{equation}\label{eqn:0-C-map-I-0}
\xymatrix@C2pc{
\CH_0(X) \ar@{->>}[r]^-{cyc_{(X,mS)}} \ar@{->>}[d]_{f^*} & F^{2}K_0(X,mS) 
\ar@{->>}[d]^{f^*} \ar[dr]^{\simeq} & \\
\CH_0(\wt{X},mE) \ar@{->>}[r]_>>>>{cyc_{(\wt{X}, mE)}} \ar@{->>}[d] & 
F^{2}K_0(\wt{X},mE)
\ar@{->>}[d] & F^{2}K_0(X) \ar@{->>}[dl]^{f^*} \\
\CH_0(\wt{X}) \ar@{->>}[r]_-{cyc_{\wt{X}}} & F^{2}K_0(\wt{X}). &}
\end{equation}
such that all arrows are surjective.
Moreover, all arrows in the top square are isomorphisms for all $m \gg 1$.
In particular, the cycle class map $\CH_0(\wt{X},mE) \to K_0(\wt{X},mE)$
is injective for all $m \gg 1$.
\end{prop}
\begin{proof}
Let $F^{2}K_0(\wt{X},mE)$ denote the image of $\CH_0(\wt{X}, mE)$ under
the cycle class map of \propref{prop:Cyc-class-mod}.
The maps on the $K$-theory side are obvious maps induced by the
pull-back $f^*:K(X, mS) \to K(\wt{X}, mE)$. The surjectivity of arrows in
the triangle on the right is immediate from the construction of the
underlying groups. The map $cyc_{(X, mS)}$ exists and is surjective by
\lemref{lem:CCS-0}. 
The maps $cyc_{(\wt{X}, mE)}$ and $cyc_{\wt{X}}$ clearly
are surjective.  
The map $\CH_0(\wt{X},mE) \to \CH_0(\wt{X})$ is the forgetful map which
is surjective by the moving lemma for $\CH_0(\wt{X})$. 

It is clear that there is a pull-back map $f^*: \sZ_0(X) \to
\sZ_0(\wt{X}, mE)$ which is surjective. To show that it preserves rational
equivalences, let $C \inj X$ be an integral
curve not meeting $X_{\rm sing}$ and let $h \in k(C)^{\times}$.
Let $\Gamma_h \inj C \times \P^1 \inj X \times \P^1$ be the graph of the
function $h: C \to \P^1$. It is then clear that $\Gamma_h \cap (X_{\rm sing}
\times \P^1) = \emptyset$. In particular, $f^{-1}(\Gamma_h) \cap
(E \times \P^1) = \emptyset$. This shows that $[\Gamma_h] \in
\sZ_1(\wt{X}, mE)$ is an admissible 1-cycle such that
\[
f^*({\rm div}(h)) = f^*([h^*(0)] - [h^*(\infty)]) =
f^*(\partial_{0}([\Gamma_h]) - \partial_{\infty}([\Gamma_h]))
= (\partial_{0} - \partial_{\infty})([\Gamma_h]).
\]
This shows that $f^*({\rm div}(h)) \subset \sR_0(\wt{X}, mE)$
and yields the pull-back $f^*: \CH_0(X) \surj \CH_0(\wt{X}, mE)$.
It is clear from the constructions that the diagram~\ref{eqn:0-C-map-I-0}
commutes. This proves the first part of the proposition.

We now prove the second part. \propref{prop:Surface-0} says that 
$cyc_{(X, mS)}$ is an isomorphism
and so is the map $F^2K_0(X, mS) \to F^2K_0(X)$.
On the other hand, it follows from \cite[Theorem~1.1]{KSri} that the
map $F^2K_0(X) \to F^2K_0(\wt{X})$ factors through $F^2K_0(X) \to
F^2K_0(\wt{X}, mE) \to F^2K_0(\wt{X})$ and the map
$F^2K_0(X) \to F^2K_0(\wt{X}, mE)$ is an isomorphism for all $m \gg 1$.
It follows that the right vertical arrow in the top square is an isomorphism 
for all $m \gg 1$. A simple diagram chase now shows that all arrows in the top
square are isomorphisms for all $m \gg 1$.       
\end{proof}

\begin{prop}\label{prop:Nor-fin}
Let $X$ be a connected projective surface over $k$ with isolated 
singularities. Then $A_0(X)$ is finite.
\end{prop}
\begin{proof}
Using \propref{prop:Surface-0}, we can assume that $X$ is normal
and hence integral.
Let $f: \wt{X} \to X$ be a resolution of singularities such that
the reduced exceptional divisor $E$ is strict normal crossing.
Using \propref{prop:RES-Norm} and the finiteness of $A_0(\wt{X})$ 
(see \cite[Theorem~1]{KS}), we reduce to showing that
${\rm Ker}(F^2K_0(\wt{X}, mE) \surj F^2K_0(\wt{X}))$ is finite.
By \cite[Lemma~2.2]{KSri}, this kernel is contained in the image of the
boundary map
$\partial:SK_1(mE) \to K_0(\wt{X}, mE)$. Hence, it suffices to prove 
inductively that $SK_1(mE)$ is finite for $m \ge 1$. 

Let $U \inj E$ be a dense open subscheme of $E$ containing
the singular locus of $E$ and set $Z = (E \setminus U)_{\rm red}$. The
localization sequence of Thomason-Trobaugh yields an exact
sequence $K_1(E \ \mbox{on} \ Z) \to SK_1(E) \to SK_1(U)$.
Since $U$ is affine, $SK_1(U)$ vanishes by \cite[Proposition~2.1]{Nestler}.
Since $Z \inj E_{\rm reg}$, the excision isomorphism
$K_1(E \ \mbox{on} \ Z) \xrightarrow{\simeq} K_1(E_{\rm reg} \ \mbox{on} \ Z)$
(see \cite[Proposition~3.19]{TT}) and the localization fiber sequence
$K(Z) \to K(E_{\rm reg}) \to K(E_{\rm reg}\setminus Z)$ show that
the map $K_1(E \ \mbox{on} \ Z) \to K_1(Z)$ is an isomorphism.
Since $K_1(Z) \simeq \sO^{\times}_Z$ is finite, the surjection
$K_1(E \ \mbox{on} \ Z) \surj SK_1(E)$ shows that $SK_1(E)$ is finite.

Assume now that $SK_1(mE)$ is finite for some $m \ge 1$. 
Let $S = \{E_1, \cdots , E_r\}$ denote the set of irreducible components
of $E$ and let $S_{ij} = E_i \cap E_j$.
By \cite[\S~6.4]{Milnor}, there is an exact sequence of Zariski sheaves
\[
{\sK}_{2,(m+1)E} \to {\oplus}_{E_i \in  S} {\sK}_{2,(m+1)E_i}
\to {\oplus}_{E_i \in S} {\oplus}_{P \in S_{ij}} {\sK}_{2,P_{(m+1)}}
\to 0
\]
on $(m+1)E$, where $P_{(m+1)}$ is the closed subscheme of $\wt{X}$ with support 
$P \in S_{ij}$ whose local ring is analytically isomorphic to
${l[[x,y]]}/{(x^{m+1}, y^{m+1})}$ for some finite field extension $k \inj l$. 
The surjectivity of the last map can be 
easily checked locally. This yields a long cohomology exact sequence
\[
{\underset{E_i \in S}\oplus} \
{\underset{P \in S_{ij}}\oplus} \ K_2(P_{(m+1)})
\to H^1((m+1)E, \sK_{2,(m+1)E}) \to 
{\underset{E_i \in S}\oplus} \
H^1((m+1)E_i, {\sK}_{2,(m+1)E_i}) \to 0.
\]
The surjection $\sO^{\times}_{P_{(m+1)}} {\underset{\Z}\otimes} 
\ \sO^{\times}_{P_{(m+1)}} \surj K_2(P_{(m+1)})$ 
implies that the first term of this exact sequence
is finite. 

Using the isomorphism
$SK_1((m+1)E) \xrightarrow{\simeq} H^1((m+1)E, \sK_{2,(m+1)E})$
(see \cite[Lemma~2.3]{KSri}), it suffices to show that each
$SK_1((m+1)E_i)$ is finite. We can thus assume that
$E$ is smooth.

The short exact sequence of sheaves
\[
\sK_{2,((m+1)E, mE)} \to \sK_{2,(m+1)E} \to \sK_{2, mE} \to 0
\]
yields an exact sequence
\[
H^1((m+1)E, \sK_{2,((m+1)E, mE)}) \to SK_1((m+1)E) \to SK_1(mE) \to 0.
\]
Since $k$ is finite, it follows from \cite[Theorem~7.1]{KSri} that  
$\sK_{2,((m+1)E, mE)}$ is a quotient of a coherent sheaf on $E$.
In particular, $H^1((m+1)E, \sK_{2,((m+1)E, mE)})$ is finite.
We conclude by induction that $ SK_1((m+1)E)$ is finite and this
finishes the proof of the proposition.
\end{proof}

{\sl Proof of \thmref{thm:KES-surf}:}
Let $U$ be the regular locus of a normal projective surface $X$ over $k$.
Since $C(U)$ does not depend on the choice of a compactification of $U$,
we take this compactification to be $\wt{X}$, where $f:\wt{X} \to X$ is
a resolution of singularities of $X$. Note that such a resolution always
exists for surfaces.

If $D \subsetneq \wt{X}$ is an effective Cartier divisor with
support $|D| \subseteq E$, then $mE - D$ must be an effective Cartier
divisor some $m \gg 1$. This implies that the canonical map
$C(U) \to {\underset{m}\varprojlim} \ \CH_0(\wt{X}, mE)$ is an
isomorphism. 
Moreover, it follows from \propref{prop:RES-Norm} that
the maps $\CH_0(X) \to C(U) \to \CH_0(\wt{X}, mE)$ are isomorphisms
for all $m \gg 1$.

Since the map $\sZ_0(X) \xrightarrow{f^*} 
\sZ_0(\wt{X}, mE)$ is identity, the reciprocity map for $X$ defines
a similar map $\sZ_0(\wt{X}, mE) \to \pi^{\rm ab}_1(U)$.
It follows now from \thmref{thm:Main-Surf} that this induces
a reciprocity map $\rho_U: C(U) \to \pi^{\rm ab}_1(U)$ and  
an isomorphism $C_0(U) \xrightarrow{\simeq}
\pi^{\rm ab}_1(U)^0$. Moreover, it follows from Propositions~\ref{prop:RES-Norm} 
and 
~\ref{prop:Nor-fin} that $C_0(U)$ is finite. This proves \thmref{thm:KES-surf}.
$\hspace*{10cm} \hfill \square$

\section{Class field theory in higher dimensions}\label{sec:CFTH}
In this section, we shall prove \thmref{thm:CFT-Sing} using an inductive
argument on the dimension of the scheme.
Let $X$ be a geometrically integral projective scheme over $k$ of dimension
$d \ge 2$ which is regular in codimension one and let $j: U \inj X$ denote the 
inclusion of the regular locus of $X$. 
We can define the reciprocity map $\theta_X: \sZ_0(X) \to \pi^{\rm ab}_1(U)$ 
exactly as we did for surfaces in \thmref{thm:Main-Surf}. 
We shall prove that this
map kills rational equivalences using the following steps.

\begin{lem}\label{lem:Rec-Cov}
Let $f:X' \to X$ be a morphism of projective schemes
such that $f^{-1}(X_{\rm sing}) \subseteq X'_{\rm sing}$. Then there is 
a commutative diagram
\begin{equation}\label{eqn:Rec-Cov-0}
\xymatrix@C1pc{
\sZ_0(X') \ar[r]^{\theta_{X'}} \ar[d]_{f_*} &  \pi^{\rm ab}_1(X'_{\rm reg}) 
\ar[d]^{f_*} \\
\sZ_0(X) \ar[r]_{\theta_{X}} &  \pi^{\rm ab}_1(X_{\rm reg}).}
\end{equation}
\end{lem}
\begin{proof}
It is clear from our assumption that $f(X'_{\rm reg}) \subseteq X_{\rm reg}$.
In particular, the vertical arrows in ~\eqref{eqn:Rec-Cov-0} are defined.
Let $x' \in  X'_{\rm reg}$ be a closed point with residue field $k(x')$ and let
$f(x') = x$ with residue field $k(x)$. We then have a diagram
\[
\xymatrix@C1pc{
\sZ_0(k(x')) \ar[rr] \ar[dd] \ar[dr]^{\iota'_*} & & \pi^{\rm ab}_1(k(x')) \ar[dd] 
\ar[dr] & \\
& \sZ_0(X') \ar[dd] \ar[rr] & & \pi^{\rm ab}_1(X') \ar[dd] \\
\sZ_0(k(x)) \ar[rr] \ar[dr] & & \pi^{\rm ab}_1(k(x)) \ar[dr] & \\
& \sZ_0(X) \ar[rr] & & \pi^{\rm ab}_1(X).}
\]  
It suffices to show that $f_* \circ \theta_{X'}(\iota'_*(1)) =
\theta_X \circ f_*(\iota'_*(1))$ in order to prove the lemma.
In other words, we need to show that all faces of the cube commute.

The left and the right faces commute by the covariant functoriality of
the fundamental group and the group of 0-cycles. The top and the bottom
faces commute by the definition of the reciprocity map. It thus suffices to
show that the back face commutes. In other words, we need to show that the
diagram
\[
\xymatrix@C1pc{
\Z \ar[r]^<<<{\theta'} \ar[d]_{d} & {\rm Gal}({\ov{k}}/L) \ar[d]^{f_*} \\
\Z \ar[r]_<<<{\theta} & {\rm Gal}({\ov{k}}/K)}
\]
commutes, where $K = k(x)$, $L = k(x')$, $d = [L:K]$ and
$f: \Spec(L) \to \Spec(K)$. 
The commutativity of this diagram is then equivalent to knowing that
$f_*(F_L) = dF_K$, where $F$ is the Frobenius automorphism of $\ov{k}$.
But this is obvious if we write $K = \F_q$ and $L = \F_{q^d}$.
\end{proof}

\begin{lem}\label{lem:PF-sing}
Let $f: X' \to X$ be a finite morphism of irreducible quasi-projective 
schemes over $k$ of dimension $d \ge 2$ with isolated singularities. 
Assume that $f^{-1}(X_{\rm sing}) \subseteq X'_{\rm sing}$. Then there
is a push-forward map $f_*: \CH_0(X') \to \CH_0(X)$.
\end{lem}
\begin{proof}
Our assumption implies that there is a push-forward map
$f_*: \sZ_0(X') \to \sZ_0(X)$. Let $C' \subsetneq X'$ be an irreducible curve
let $h \in k(C')^{\times}$. As $\dim(X') \ge 2$ and 
it has only isolated singularities, we must have $C' \cap X'_{\rm sing} =
\emptyset$. Setting $C = f(C')$, we see that $C' \subsetneq X_{\rm reg}$
is an irreducible curve. Let $N: k(C') \to k(C)$ denote the norm map.
It follows now from \cite{Fulton} that 
$f_*({\rm div}(h)) = {\rm div}(N(h)) \in \sR_0(X)$.
In particular, $f_*$ preserves the groups of rational equivalences
and yields the desired map on the Chow groups. 
\end{proof}

\begin{prop}\label{prop:Kill-Rat}
Let $X$ be a integral projective scheme over $k$ of dimension
$d \ge 2$ which is regular in codimension one and let $j: U \inj X$ denote the 
inclusion of the regular locus of $X$. Then the map
$\theta_X: \sZ_0(X) \to \pi^{\rm ab}_1(U)$ kills the group of rational
equivalences to induce a map
\[
\theta_X: \CH_0(X) \to \pi^{\rm ab}_1(U).
\]
which restricts to a map $\theta^0_X: A_0(X) \to \pi^{\rm ab}_1(U)^0$.
\end{prop} 
\begin{proof}
If $d =2$, then the proposition follows \thmref{thm:Main-Surf}.
We now let $d \ge 3$ and assume that
$\theta_X$ kills rational equivalences if $2 \le \dim(X) \le d-1$.

Let $\iota: C \inj X$ be a closed irreducible curve such that $C \subsetneq U$
and let $f \in k(C)^{\times}$ be a rational function. 
Since $X$ is irreducible and regular in codimension one, it 
suffices to show (see the definition of Cartier curves and 
\lemref{lem:RE}) that $\theta_X \circ \iota_* ({\rm div}(f)) = 0$
in order to show that $\theta_X$ kills rational equivalences.

We assume that $C$ is singular. The other case is easier as we will see.
We claim that there exists a finite sequence of point blow-ups 
$\pi: X' \to X$ with centers in $U$ such that the strict transform $C'$ of $C$ 
is the normalization of $C$. 

To see this, we first observe that if $x \in C$ is a closed point,
then the blow-up of $C$ at $x$ is same as the strict transform of
$C$ in the blow-up of $X$ at $x$. Using this observation and using the
inclusion $C \inj U$, it suffices to show that the normalization
$C^N \to C$ is a composite of point blow-ups of $C$ with centers in 
$C_{\rm sing}$. Let $D$ denote the singular locus of $C$. 

Let us choose a point $x \in D$ and let $\pi_x: C^x \to C$ be the blow-up
of $C$ at $x$. We now notice that $\pi_x$ is a finite birational map
and dominated by the normalization $C^N \to C^x \to C$. We next observe that
$\pi_x$ can not be an isomorphism. This is because of the fact (which one
can easily check) that $\pi_x$ will be an isomorphism if and only if
$x \in C$ is a Cartier divisor. But as $C$ is a curve, this will happen if
and only if $C$ is regular at $x$, which we have assumed it is not.
Since all such blow-ups are dominated by $C^N$ which is finite over $C$, it
follows that this process will end at $C^N$ (otherwise, $\sO_{C^N}$ will not
be a coherent $\sO_C$-module). This proves our claim.

Using the above claim, we choose a composite of point blow-ups
$\pi: X' \to X$ with centers in $U$ such that the strict transform
$C'$ of $C$ is the normalization of $C$.  
Notice that $\pi: X' \to X$ is an isomorphism
over a neighbourhood of $X_{\rm sing}$. In particular, the following hold.
\begin{enumerate}
\item
$X'$ is irreducible and regular in codimension one.
\item
$U' := X'_{\rm reg} = \pi^{-1}(U)$.
\item
$C'$ is smooth and $C' \subsetneq U'$. 
\item
If we let $\nu: C' \to C$ denote the restriction of $\pi$, then we can think
of $f$ as a rational function on $C'$ such that ${\rm div}((f)_C) =
\nu_*({\rm div}((f)_{C'}))$ (see \cite{Fulton}).
\end{enumerate}

We embed $X' \inj \P^N_k$ as a closed subscheme for some $N \gg 1$.
By \cite[Theorem~1.1]{Poonen-3}, we can find a hypersurface section
$H \subsetneq X'$ containing $C'$ such that $H \cap U'$ is smooth
and $X'_{\rm sing} \not\subset H$.
By \cite[Corollary~1.4]{Poonen-2}, we can also assume that $H$ is irreducible.

Since $H$ is a hypersurface section of  $X'$, 
it must be singular along $X'_{\rm sing}$.
In particular, $H_{\rm sing} = X'_{\rm sing} \cap H$. As 
$X'_{\rm sing} \not\subset H$, we see that $\dim(H \cap X'_{\rm sing}) =
\dim(X'_{\rm sing}) - 1 \le d-3$. We conclude that
$H$ is a hypersurface section of $X'$ such that $H_{\rm sing}$ has
dimension at most $d-3$, $H_{\rm reg} = H \cap U'$ and
$C' \subsetneq H_{\rm reg}$. In particular, $H$ is regular in codimension one.
Let $\pi': H \inj X' \to X$ denote the composite map. 
We now apply \lemref{lem:Rec-Cov} to $\pi':H \to X$ to get
\[
\theta_X({\rm div}(f)_C) = \theta_X \circ \pi'_*({\rm div}(f)_{C'}) =
\pi'_* \circ \theta_{H}({\rm div}(f)_{C'}) {=}^{\dagger} 0,
\]
where ${=}^{\dagger}$ follows by the induction hypothesis.
This shows that $\theta_X$ kills rational equivalences.
Finally, it is clear from the construction of $\theta_X$ 
that it induces a map $\theta^0_X: A_0(X) \to \pi^{\rm ab}_1(U)^0$.
This finishes the proof. 
\end{proof}

\subsection{Conclusion of the Proof of  \thmref{thm:CFT-Sing}}
\label{sec:Final}
We now complete the proof of \thmref{thm:CFT-Sing}. We are only left with
showing that $\theta^0_X$ is an isomorphism of finite groups if $X$
has only isolated singularities. It will then follow from
~\eqref{eqn:B-formula-0} that $\theta_X$ is injective with dense image. 
We shall prove the isomorphism of $\theta^0_X$ with the aid of the
following Bertini type theorem over finite fields, which is 
our third main step.

\begin{thm}$($\cite[Theorem~1.3]{Poonen-2}$)$\label{thm:Poonen}
Let $U \inj \P^N_k$ be a smooth and geometrically integral locally closed 
subscheme of dimension at least three.
Let $T_1, T_2 \subsetneq \P^N_k$ be two disjoint finite sets of closed points. 
Then for all sufficiently large $d \gg 1$, there exist hypersurfaces 
$H \subsetneq \P^N_k$ of degree $d$ such that $H \cap T_1 = \emptyset$,
$T_2 \subsetneq H$ and $H \cap U$ is a smooth and geometrically integral
hypersurface section of $U$.
\end{thm}
\begin{proof}
For a closed point $P \in \P^N_k$, let $\wh{\sO}_P$ denote the completion
of the local ring of $\P^N_k$ at $P$. For $P \in T_1$, let
$\sU_P = ({\wh{\sO}_P})^{\times}$ and for $Q \in T_2$, let
$\sU_Q = \fm_Q$, the maximal ideal of $\wh{\sO}_Q$.
Then for each $P \in T_1 \cup T_2$, the set $\sU_P \subsetneq \wh{\sO}_P$ is
a union of cosets of $\fm_P$. Moreover, $\fm_P$ has finite index in
$\wh{\sO}_P$. We can now apply \cite[Theorem~1.3]{Poonen-2} (see its 
proof) to conclude that for all sufficiently large $d \gg 1$, there exist
hypersurfaces $H = H(f) \subsetneq \P^N_k$ which have the property that
$x_{j_P}^{-e_P}f \in \sU_P$ for each $P \in T_1 \cup T_2$ and $H \cap U$ is 
smooth. Here, $x_{j_P}$ is a non-vanishing coordinate of $\P^N_k$ at $P$ so that
$x_{j_P}^{-e_P}f \in \sO_P$. Our conditions at the points of
$T_1 \cup T_2$ imply that $T_1 \cap H = \emptyset$ and
$T_2 \subsetneq H$.
Using \cite[Corollary~1.4]{Poonen-2}, we can also assume that $H \cap U$
is geometrically integral. This finishes the proof.
\end{proof}

{\sl Proof of \thmref{thm:CFT-Sing}:}
Let $X$ be a geometrically integral projective scheme over $k$ of 
dimension $d \ge 2$ with isolated singularities.
If $d =2$, \thmref{thm:CFT-Sing} follows from \thmref{thm:Main-Surf}
and \propref{prop:Nor-fin}.
We now let $d \ge 3$ and assume that $\theta^0_X$ is
an isomorphism of finite groups if $2 \le  \dim(X) \le d-1$.

Let $p:X' \to X$ denote the normalization map and set $S' = p^{-1}(X_{\rm sing})$,
and $U = X_{\rm reg}$. We can then identify $U$ as 
the open subset $X' \setminus S'$. Notice that $S'$ is a finite set since
$X$ has only isolated singularities. 

We now embed $X' \inj \P^N_k$ as a closed subscheme for some $N \gg 1$. 
To show that $\theta^0_X$ is injective, let $\alpha \in A_0(X)$ be a 0-cycle 
such that $\theta^0_X(\alpha) = 0$. Since the support of $p^*(\alpha)$ and
$S'$ are two disjoint finite subsets of closed points in $X'$,
we can apply \thmref{thm:Poonen} to get a geometrically integral 
hypersurface section $Y = H \cap X'$ of $X'$ which is disjoint from 
$S'$, contains the support of $p^*(\alpha)$ and $H \cap U$ is smooth.
In particular, $Y \subsetneq U$ is actually a closed subscheme of $X$
and $\alpha$ is a 0-cycle of degree zero on $Y$. 

Applying Lemmas~\ref{lem:Rec-Cov} and ~\ref{lem:PF-sing} 
to the inclusion $\pi:Y \inj U \inj X$, we get a commutative diagram:
\begin{equation}\label{eqn:Final-0}
\xymatrix@C1pc{
A_0(Y) \ar[r]^{\theta^0_Y} \ar[d]_{\pi_*} & \pi^{\rm ab}_1(Y)^0 \ar[d]^{\pi_*} 
\\
A_0(X) \ar[r]_{\theta^0_X} & \pi^{\rm ab}_1(U)^0.}
\end{equation} 
 
We claim that the map $\pi_*: \pi^{\rm ab}_1(Y)^0 \to \pi^{\rm ab}_1(U)^0$
is an isomorphism. Since $X'$ is normal of dimension at least two,
it has depth at least two at each of its closed points.
Since $Y$ is a hypersurface section of $X'$ which is contained in $U$
(which is an open subset of $X'$)
and since $d \ge 3$, we see that $X'$ has depth at least three at each
closed point of $Y$. We conclude from 
\cite[Expos{\'e}~XII, Corollaire~3.5]{SGA-2} that the map
$\pi_1(Y) \to {\underset{W}\varprojlim} \ \pi_1(W)$ is an isomorphism,
where the limit is taken over all open neighborhoods of $Y$ contained
in $U$. It therefore suffices to show that the map $\pi_1(W) \to \pi_1(U)$
is an isomorphism for each such $W$.

Set $Z' = U \setminus W$ and let $Z$ denote its closure in $X'$. Then 
$Z \inj \P^N_k$ is a closed subscheme disjoint from $Y = X' \cap H$ and
hence from $H$. But this can happen only if $\dim(Z) \le 0$.
If $Z = \emptyset$, we are done. If $\dim(Z) = 0$, then 
$W \inj U$ is an open immersion of smooth schemes such that
$U \setminus W$ has codimension $d \ge 2$. It follows from the
Zariski-Nagata theorem (see \cite[Expos{\'e}~X, Th{\'e}or{\`e}me~3.1]{SGA-1})
that the map $\pi_1(W) \to \pi_1(U)$ must be an isomorphism and this
proves the claim.

Since $Y$ is smooth projective and geometrically integral, it follows
from \cite[Theorem~1]{KS} that $\theta^0_Y$ is an isomorphism.
Combining this with the above claim, it follows that the map
$\pi_* \circ \theta^0_Y$ is an isomorphism. It follows that 
$\alpha$ dies as a class in $A_0(Y)$ and hence it dies in $A_0(X)$.
This shows that $\theta^0_X$ is injective. Furthermore, the
isomorphism of $\pi_* \circ \theta^0_Y = \theta^0_X \circ \pi_*$ also
shows that $\theta^0_X$ is surjective as well.

Lastly, the finiteness of $A_0(X)$ and $\pi^{\rm ab}_1(U)^0$ follows because
we have shown that all arrows in ~\eqref{eqn:Final-0} are isomorphisms
and the groups on the top are finite by induction.
This completes the proof of \thmref{thm:CFT-Sing}.
$\hfill \square$

\vskip .4cm

{\sl Proof of \corref{cor:finite-Chow}:}
The case of curve is easy because the maps $A_0(X) \to A_0(X^N) \to 
{\rm Alb}_X(k)$ are surjective maps of finite groups. If $X$ is a
surface with isolated singularities, we can assume it to be normal by
\propref{prop:Surface-0}. If $\wt{X} \to X$ is a resolution of singularities,
we have the maps $A_0(X) \to A_0(\wt{X}) \to {\rm Alb}_X(k)$.
The first map is clearly surjective and the second map is surjective by
\cite[Proposition~9]{KS}. The finiteness of these groups is already shown
before.

If $d \ge 3$, we can find a general hypersurface section $Y = H \cap X$ as
in the proof of \thmref{thm:CFT-Sing}. It is shown in
~\eqref{eqn:Final-0} that the map $A_0(Y) \to A_0(X)$ is an isomorphism.
On the other hand, it follows from \cite[Chapter~8, \S~2, Theorem~5]{Lang-2}
that the map ${\rm Alb}_Y \to {\rm Alb}_X$ is an isomorphism.
The induction hypothesis now finishes the proof.
$\hfill \square$

\vskip .4cm

{\sl Proof of \thmref{thm:B-formula}:}
Let $X$ be as in \thmref{thm:B-formula}. The $d \le 1$ case follows from
\propref{prop:0-cyc-curve}. So we assume $d \ge 2$.
\propref{prop:Milnor-res-*} says that the map
$H^d(X, \sK^M_{d,X}) \to H^d(X, \sK_{d,X})$ is an isomorphism. It also says that
there is a commutative diagram with exact rows:
\begin{equation}\label{eqn:B-formula-0}
\xymatrix@C1pc{
0 \ar[r] & A_0(X) \ar[r] \ar[d] & \CH_0(X) \ar[r] \ar[d] & \Z \ar@{=}[d] 
\ar[r] & 0 \\
0 \ar[r] & H^d(X, \sK^M_{d,X})^0 \ar[r] \ar[d] & H^d(X, \sK^M_{d,X}) \ar[r] 
\ar[d] &  \Z \ar@{^{(}->}[d] \ar[r] & 0 \\
0 \ar[r] & \pi^{\rm ab}_1(U)^0 \ar[r] & \pi^{\rm ab}_1(U) \ar[r] & 
\wh{\Z} \ar[r]
& 0.}
\end{equation}

Note that the maps $\CH_0(X) \to H^d(X, \sK^M_{d,X}) \to \Z$ are the
composite $\CH_0(X) \to H^d(X, \sK^M_{d,X}) \to \CH^F_0(X) \to \Z$
(see \propref{prop:Milnor-res-*}), where the last map is the degree map
induced by the projective push-forward $\CH^F_0(X) \to \CH^F_0(\Spec(k))$.  

\thmref{thm:CFT-Sing} says that the the composite vertical arrow on the left
is an isomorphism. It follows from \corref{cor:Surface-0-*1}
and its obvious higher dimensional analogue that the map
$H^d(X, \sK^M_{d,(X, S)}) \to H^d(X, \sK^M_{d,X})$ is an isomorphism for
all closed subschemes $S \inj X$ supported on $X_{\rm sing}$.
Combining this with \cite[Theorem~6.2]{Raskind}, we see that
the bottom vertical arrow on the left is also an isomorphism.
It follows that the left vertical arrow on the top is an isomorphism
and this implies that the middle vertical arrow on the top must also be
an isomorphism.
$\hfill \square$
 
\vskip .2cm

\noindent\emph{Acknowledgements.}
The author would like to thank N. Fakhruddin for some valuable conversation 
during the preparation of this note.

\end{document}